\documentclass[a4paper,abstracton]{scrartcl}

\usepackage{amsfonts,amsthm,amssymb,amsmath}
\usepackage[latin1]{inputenc}
\usepackage[T1]{fontenc}
\usepackage{graphicx}
\usepackage[inline]{enumitem}
\usepackage{hyperref}
\usepackage{tikz}
\usepackage{tkz-berge}

\newtheorem{thm}{Theorem}[section]
\newtheorem{lem}[thm]{Lemma}
\newtheorem{prp}[thm]{Proposition}

\theoremstyle{definition}

\newtheorem{exmp}[thm]{Example}

\newtheorem{ass}[thm]{Assumption}
\newtheorem{clm}{Claim}

\newcommand{\Aut}{\operatorname{Aut}}

\newcommand{\aed}{\color{black}}
\newcommand{\bed}{\color{black}}
\newcommand{\ced}{\color{black}}
\newcommand{\zed}{\color{black}}

\title{On symmetries of edge and vertex colourings of graphs}
\author{Florian Lehner\footnote{Address: Mathematics Institute, University of Warwick, U.K. \newline Florian Lehner was supported by the Austrian Science Fund (FWF), grant J 3850-N32} and Simon M. Smith\footnote{Address: Charlotte Scott Research Centre for Algebra, School of Mathematics and Physics, University of Lincoln, U.K.}}

\begin{document}
\maketitle

\begin{abstract}
Let $c$ and $c'$ be edge or vertex colourings of a graph $G$. We say that $c'$ is less symmetric than $c$ if the stabiliser (in $\Aut G$) of $c'$ is contained in the stabiliser of $c$. 

We show that if $G$ is not a bicentred tree, then for every vertex colouring of $G$ there is a less symmetric edge colouring with the same number of colours. On the other hand, if $T$ is a tree, then for every edge colouring there is a less symmetric vertex colouring with the same number of edges.

Our results can be used to characterise those graphs whose distinguishing index is larger than their distinguishing number.
\end{abstract}
\section{Introduction}

This paper concerns the symmetries  of colourings of graphs. A large number of research papers have been written on this topic; mostly, these have focused on finding colourings with few symmetries and a small number of colours. In particular, the \emph{distinguishing number} $D(G)$ of a graph is the least number of colours such that there is a vertex colouring which is not preserved by any non-identity automorphism. Motivated by a recreational mathematics problem, this notion was introduced by Albertson and Collins in 1996 \cite{albertsoncollins-distinguishing} and has since received considerable attention.

Recently, Kalinowski and Pil\'sniak~\cite{kalinowski-distindex} suggested the following edge version. The \emph{distinguishing index} $D'(G)$ of a graph $G$ is the smallest number of colours such that there is an edge colouring which is not preserved by any non-identity automorphism. Many results about distinguishing numbers hold for distinguishing indices as well, sometimes with almost identical proofs (see \cite{Broere-distindex,imrich-distindex,kalinowski-distindex}, for example). Furthermore, there are problems such as Tucker's Infinite Motion Conjecture \cite{tucker-infinitemotion} that are still wide open for vertex colourings, but whose edge colouring version has a relatively simple proof \cite{lehner-edgemotion}. This suggests that finding edge colourings with few symmetries is generally easier than finding such vertex colourings. This suggestion is further supported by the following result. It was first proved for finite graphs in \cite{kalinowski-distindex}. See \cite{imrich-distindex} for an extension to infinite graphs and \cite{lehner-edgemotion} for an alternative proof for both finite and infinite graphs.

\begin{thm}
\label{thm:distindexdistnumber}
If $G$ is a connected graph of order at least $3$, then $D'(G) \leq D(G) + 1$.
\end{thm}

\ced It was shown in \cite{imrich-distindex} that  Theorem~\ref{thm:distindexdistnumber} remains true even \bed if $D(G)$ is an arbitrary infinite cardinal. Since $\alpha = \alpha + 1$ for any infinite cardinal $\alpha$, this implies that for any graph $G$ with infinite distinguishing number we have $D'(G) \leq D(G)$.\zed \\

In this paper we 
\ced thoroughly investigate the relationship between $D'(G)$ and $D(G)$. \zed For this purpose, we say that an (edge or vertex) colouring $c'$ is \emph{less symmetric} than an (edge or vertex) colouring $c$, if the stabiliser of $c'$ (that is, the setwise stabiliser of those edges or vertices coloured $c'$) is a subgroup of the stabiliser of $c$. With this notion we have the following results.

\bed
\begin{thm}
\label{thm:vertextoedgecolouringinfinite}
Let $G$ be a connected graph and let $c$ be a $k$-vertex colouring of $G$ where $k$ is an infinite cardinal. Then there is a $k$-edge colouring $c'$ of $G$ which is less symmetric than $c$.
\end{thm}
\zed

\begin{thm}
\label{thm:vertextoedgecolouring}
Let $G$ be a connected graph and let $c$ be an arbitrary $k$-vertex colouring of $G$ \bed with $k \in \mathbb{N}$. \zed  If $G$ is not a bicentred tree, then there is an $k$-edge colouring $c'$ which is less symmetric than $c$.
\end{thm}

In particular, this implies that the inequality in Theorem~\ref{thm:distindexdistnumber} can only be sharp if $G$ is a bicentred tree \bed with finite distinguishing number\zed. Note that examples of trees where Theorem~\ref{thm:distindexdistnumber} holds with equality are known, \bed see \cite{kalinowski-distindex}\zed, whence Theorem~\ref{thm:vertextoedgecolouring} does not extend to all graphs. In fact, in the case of trees we show that the converse of Theorem~\ref{thm:vertextoedgecolouring} is true.

\begin{thm}
\label{thm:edgetovertexcolouring}
Let $T$ be a tree and let $c'$ be an arbitrary $k$-edge colouring of $T$. \zed Then there is a $k$-vertex colouring $c$ which is less symmetric than $c'$.
\end{thm}

Finally, as an application of our results we characterise all graphs for which Theorem~\ref{thm:distindexdistnumber} is sharp.

\section{Preliminaries}

\subsection{Basic notions and notations}

Throughout this paper $G = (V,E)$ will denote a graph with vertex set $V$ and edge set $E$. We follow \cite{diestel-buch} for notions that are not explicitly defined. Note that we do not make assumptions on the cardinality of $V$ or the degree of a vertex. We will, however, require our graphs to be simple (i.e.\ no multiple edges or loops) and undirected. The automorphism group of $G$ will be denoted by $\Aut G$, and for $\gamma \in \Aut G$ and $x \in V \cup E$ we will denote by $\gamma x$ the image of $x$ under the automorphism $\gamma$.

A \emph{vertex colouring} of $G$ is a function from $V$ to a set $C$ of colours. If $|C| = k$ we will speak of a $k$-vertex colouring and \bed if $k$ is finite, we \zed will usually assume that $C = \{1,2,\ldots,k\}$. We say that $\gamma \in \Aut G$ \emph{preserves} a vertex colouring $c$, if for every $v \in V$ ($e \in E$) we have $c(v) = c(\gamma v)$. The \emph{stabiliser} of a vertex colouring is the set of all automorphisms which preserve it. Analogous definitions can be made for edge colourings. We say that a vertex or edge colouring $c'$ is \emph{less symmetric} than another vertex or edge colouring $c$, if the stabiliser of $c'$ is contained in the stabiliser of $c$. Clearly, the relation of being less symmetric is transitive, and if $c'$ is less symmetric than $c$ and $c$ is less symmetric than $c'$, then the stabilisers of $c$ and $c'$ coincide.

A \emph{centre} of a graph is a vertex which minimises the maximal distance to other vertices. It is a well known fact that every finite tree has either one centre or two adjacent centres. In the first case we call the tree \emph{unicentred}, in the second case we call it \emph{bicentred}. Note that every automorphism of a unicentred tree fixes the centre whereas every automorphism of a bicentred tree fixes the \emph{central edge}, i.e.\ the edge connecting its two centres. 

While the notion of a centre is not always well defined for infinite graphs, we can still define an analogous notion for certain infinite trees. A \emph{ray} is a one sided infinite path, a graph is called \emph{rayless}, if it does not contain a ray. By results of Schmidt \cite{schmidt-rayless} (see \cite{halin-rayless} for an exposition in English) every rayless tree has a canonical finite subtree (its kernel) which is preserved by every automorphism. Hence we can define a centre of a rayless tree to be a centre of its kernel and call a rayless tree unicentred or bicentred, depending on whether we get a unique centre or a two adjacent centres.

\subsection{Canonical colourings}

The following way to obtain an edge colouring from a vertex colouring was introduced in \cite{lehner-edgemotion} in order to prove Theorem~\ref{thm:distindexdistnumber}. Let $G$ be a graph and let $c\colon V \to C$ be a colouring of the vertex set of $G$ with colours in $C$. Without loss of generality assume that $C$ carries the additional structure of an Abelian group\bed ---by \cite{hajnal-algebraicAC}, any set can be endowed with an Abelian group structure as long as we assume the Axiom of Choice\zed . Now we can obtain a colouring of the edge set by $e \mapsto c(u) + c(v)$ for $e = uv$. We will call such an edge colouring a \emph{canonical} edge colouring.

\begin{prp}
\label{prp:preservecanonical}
Any vertex colouring is less symmetric than the corresponding canonical edge colouring.
\end{prp}

\begin{proof}
We have to show that any automorphism $\gamma$ preserving the vertex colouring  $c$ also preserves the canonical edge colouring $c'$. \aed Let $e = uv$ be an edge, then 
\[
	c'(\gamma(e)) = c(\gamma(u)) + c(\gamma(v)) = c(u) + c(v) = c'(e).
\zed
\]
\end{proof}

\begin{lem}
\label{lem:nofixedpoint}
Let $G$ be a connected graph, let $c$ be a vertex colouring of $G$ and let $c'$ be the corresponding canonical edge colouring. If an automorphism $\gamma$ preserves $c'$ and \aed $c(v) = c(\gamma v)$ \zed for some vertex $v$, then $\gamma$ preserves $c$.
\end{lem}

\begin{proof}
Let $u$ be a neighbour of $v$ and let $e = uv$. Then $c(\gamma u) = c'(\gamma e)- c(\gamma v) = c'(e)-c(v) = c(u)$. Now use induction over the distance from $v$ to show that the same is true for every vertex of $G$.
\end{proof}

\begin{lem}
\label{lem:changecolouring}
Let $G$ be a connected graph, let $c$ be a vertex colouring and let $c'$ be the corresponding canonical edge colouring. Let $c''$ be an edge colouring with the following properties
\begin{itemize}
\item there is at least one edge $e$ such that $c''(e) \neq c'(e)$, and
\item for every edge $e=uv$ with $c''(e) \neq c'(e)$ and every $\gamma$ which preserves $c''$ we have $c(u) = c(\gamma u)$ and $c(v) = c(\gamma v)$.
\end{itemize}
Then $c''$ is less symmetric than $c$.
\end{lem}

\begin{proof}
Let $\gamma$ be an automorphism which preserves $c''$. First we show that $\gamma$ preserves $c'$, that is, $c'(e) = c'(\gamma e)$ for every edge $e = uv$. If $c''(e) = c'(e)$ and $c''(\gamma e) = c'(\gamma e)$ then this follows from the fact that $\gamma$ preserves $c''$. If $c''(e) \neq c'(e)$, then $c'(\gamma e) = c(\gamma u) + c(\gamma v) = c(u) + c(v) = c'(e)$. If $c''(\gamma e) \neq c'(\gamma e)$ we can use an analogous argument replacing $\gamma$ by $\gamma^{-1}$.

So every automorphism which preserves $c''$ also preserves $c'$. Note that from the two conditions it is clear that there is a vertex $v$ such that $c(v) = c(\gamma v)$. Hence we can apply Lemma \ref{lem:nofixedpoint} to conclude that every automorphism which preserves $c''$ also preserves $c$ whence $c''$ is less symmetric than $c$.
\end{proof}

\section{Proof of the main results}

In this section we prove Theorems \bed \ref{thm:vertextoedgecolouringinfinite}, \zed \ref{thm:vertextoedgecolouring}, and \ref{thm:edgetovertexcolouring}. The \bed proofs of Theorems \ref{thm:vertextoedgecolouringinfinite} and \ref{thm:edgetovertexcolouring} are simple applications of results \zed from the previous section. Before giving the proofs we recall the statements of the theorems.

\bed
\begingroup
\def\thethm{\ref{thm:vertextoedgecolouringinfinite}}
\begin{thm} 
Let $G$ be a connected graph and let $c$ be a $k$-vertex colouring of $G$ where $k$ is an infinite cardinal. Then there is a $k$-edge colouring $c'$ of $G$ which is less symmetric than $c$.
\end{thm}
\addtocounter{thm}{-1}
\endgroup
\begin{proof}
Let $C$ be the set of colours used by $c$ and let $a, b \notin C$. Since $C$ is infinite, we have that $|C \cup \{a,b\}| = |C|$. Let $c'$ be the canonical edge colouring corresponding to $c$ and let $v$ be a vertex with at least two neighbours. Let $c''$ be the edge colouring obtained from $c'$ by changing the colours of two edges incident to $v$ to $a$ and $b$ respectively. It is easy to see that if an automorphism preserves $c''$ then it must fix the two recoloured edges, and consequently $c''$ satisfies the conditions of Lemma \ref{lem:changecolouring}.
\end{proof}
\zed

\begingroup
\def\thethm{\ref{thm:edgetovertexcolouring}}
\begin{thm} \aed
Let $T$ be a tree and let $c'$ be an arbitrary $k$-edge colouring of $T$. Then there is a $k$-vertex colouring $c$ which is less symmetric than $c'$. \zed
\end{thm}
\addtocounter{thm}{-1}
\endgroup

The theorem follows from Proposition \ref{prp:preservecanonical} and the following lemma.

\begin{lem}
\label{lem:edgetovertexcolouring}
Let $T$ be a tree, let $c'$ be an edge colouring of $T$, and let $v$ be an arbitrary vertex. Then there is a unique vertex colouring $c$ such that $c(v) = 0$ and $c'$ is the canonical edge colouring corresponding to $c$.
\end{lem}

\begin{proof}
We inductively construct the colouring $c$ and (simultaneously) show that it is unique.
First, define $c(v) := 0$ and suppose $\bar{c}$ that is some other vertex colouring satisfying the conditions of the lemma, with $\bar{c}(v) = 0$ and $c'$ the canonical edge colouring corresponding to $\bar{c}$.

Inductively, suppose we have defined $c$ on all vertices up to distance $r$ from $v$, and that $c$ and $\bar{c}$ agree on these vertices. Now every vertex $u$ at distance $r+1$ from $v$ has a unique neighbour $w$ at distance $r$ from $v$, and we define $c(u) := c'(uw) - c(w)$. However, the colouring $\bar{c}$ cannot have $c'$ as its canonical edge colouring unless it also assigns to $u$ the colour $c'(uw) - c(w)$. Whence $c$ and $\bar{c}$ agree on vertices up to distance $r+1$ from $v$. By induction, we have constructed $c$ and shown that $c$ is unique.
\end{proof}

Besides Theorem \ref{thm:edgetovertexcolouring}, the above lemma has another implication which will be useful later.

\begin{lem}
\label{lem:colouringbijection}
Let $T$ be a tree and let $\varphi$ be the function which maps to every vertex colouring its canonical edge colouring.
\begin{enumerate}
\item If $T$ is rooted/unicentred, then $\varphi$ is a bijection between vertex colourings in which the root/centre receives colour $0$ and edge colourings.
\item If $T$ is bicentred, then $\varphi$ is a bijection between vertex colourings in which both centres receive colour $0$ and edge colourings in which the central edge has colour $0$.
\end{enumerate}
In both cases, $\varphi$ preserves stabilisers, i.e.\ the vertex colourings to which we restrict and their corresponding edge colourings have the same stabiliser.
\end{lem}

\begin{proof}
Lemma \ref{lem:edgetovertexcolouring} implies that the maps are bijective. By Proposition \ref{prp:preservecanonical}, any vertex colouring $c$ is less symmetric than $\varphi(c)$. Conversely, since any automorphism must map centres to centres, Lemma \ref{lem:nofixedpoint} implies that for the colourings satisfying the condition $\varphi(c)$ is less symmetric than $c$, hence $\varphi$ preserves stabilisers as claimed.
\end{proof}

The remainder of the section is devoted to the proof of Theorem~\ref{thm:vertextoedgecolouring}. While the proof is divided up into several cases, each of the cases follows the same rough idea: if the canonical edge colouring corresponding to $c$ is not yet less symmetric, apply some minor modifications to it in order to obtain less symmetric colouring. Except in the case where $G$ is a tree, these modifications will only affect a finite number of edges making them easy to track.

\begingroup
\def\thethm{\ref{thm:vertextoedgecolouring}}
\begin{thm}
Let $G$ be a connected graph and let $c$ be an arbitrary $k$-vertex colouring of $G$ \bed with $k \in \mathbb N$\zed . If $G$ is not a bicentred tree, then there is an $k$-edge colouring $c'$ which is less symmetric than $c$.
\end{thm}
\addtocounter{thm}{-1}
\endgroup

\subsection{A prime number of colours is enough}

We first show that it is enough to prove Theorem~\ref{thm:vertextoedgecolouring} when the number of colours is a prime. Indeed, assume that the theorem was true for all primes and assume that we are given a colouring with $k = p_1p_2\cdots p_r$ colours. Then we can translate this colouring into a colouring $c$ with colours in $\mathbb Z_{p_1} \times \cdots \times \mathbb Z_{p_r}$. Let $c_i$ be the projection of this colouring onto the $\mathbb Z_{p_i}$ component. 

For every $i$ find an edge colouring $c_i'$ such that every automorphism preserving $c_i'$ also preserves $c_i$. Define $c' = c_1' \times \cdots \times c_r'$. Then an automorphism preserving $c'$ has to preserve every $c_i'$, hence it also preserves every $c_i$ and thus also $c$.

From now on, we will generally make the following assumption:
\begin{ass} \label{ass:general}
The colouring $c$ uses colours in $\mathbb Z_p$ where $p$ is a prime.
\end{ass}

\subsection{The tree case}

\begin{clm}
Theorem~\ref{thm:vertextoedgecolouring} is true for trees.
\end{clm}
\begin{proof}
If $T$ is a bicentred tree there is nothing to show. If $T$ has a unique central vertex, then Lemma~\ref{lem:colouringbijection} (potentially after swapping colours) implies Theorem~\ref{thm:vertextoedgecolouring}.

Hence we may assume that $T$ contains a ray. If $T$ does not contain a double ray, then by results from \cite{halin-automorphism_oneended} every automorphism must fix some vertex. We can again apply Lemma \ref{lem:nofixedpoint} to conclude that any automorphism preserving the canonical edge colouring also preserves $c$.

Finally assume that $T$ contains double rays. Define the \emph{spine} $S$ of $T$ as the subgraph induced by all vertices which lie on a double ray. Clearly, $S$ is setwise fixed by every automorphism of $T$. Observe that $S$ is a leafless tree. Pick an arbitrary root $r$ of $S$ and define an edge colouring $c_s'$ on $S$ as follows. 

Assign colour $0$ to all edges incident to $r$ and colour $1$ to all edges at odd distance from $r$ (i.e.\ edges connecting vertices at distance $2n-1$ \aed from $r$ \zed to vertices at distance $2n$). For an edge at even distance $2n$ from $r$ define the \emph{ancestor} of $e$ as the the unique vertex $v$ at distance $n$ from $r$ which lies on a path connecting $r$ to $e$. Now colour every edge at even distance from $r$ by the colour $c(v)$ of its ancestor.

We claim that any automorphism $\gamma$ which preserves $c_s'$ on $S$ also preserves $c$ on $S$. Indeed, any such $\gamma$ must fix $r$ because it is the only vertex for which all incident edges are coloured $0$. Now if $\gamma$ moves a vertex $v$ to a vertex $w$, then it also moves the edges whose ancestor is $v$ to edges whose ancestor is $w$. By definition of $c_s'$ this is only possible if $c(v) = c(w)$.

\aed Now consider the edge colouring $c''$, where $S$ is coloured according to $c_s'$ and the remaining edges according to the canonical edge colouring $c'$ of $c$. If $\gamma$ preserves $c''$ then it preserves $c$ on $S$ and therefore, by Proposition~\ref{prp:preservecanonical}, it preserves $c'$ on $S$. Hence $\gamma$ preserves $c'$. 
Since $\gamma$ preserves $c$ on $S$,
we can apply Lemma~\ref{lem:nofixedpoint} and deduce that $\gamma$ preserves $c$. \zed
\end{proof}

\subsection{Graphs containing cycles}

In this section assume that $G$ is a graph containing a cycle and $c$ is a \aed vertex \zed colouring of $G$ satisfying \aed Assumption~\ref{ass:general} \zed. Furthermore we will always assume that $c'$ is the canonical edge colouring corresponding to $c$. We start with some observations on $c'$. Let $W = e_1e_2\dots e_l$ be the edge sequence of a closed walk \aed starting at a vertex $v$.\zed
\begin{enumerate}[label=(O\arabic*)]
\item \label{obs:even} If $l$ is even, then $\sum_{i=1}^l (-1)^i c'(e_i) = 0$.
\item \label{obs:odd} If $l$ is odd, then $\sum_{i=1}^l (-1)^i c'(e_i) = - 2 c(v)$.
\end{enumerate}

Note that sum in the second observation evaluates to $0$ for $p=2$. If $p$ is odd, however, this sum can be used to find $c(v)$. Consequently, any automorphism which preserves $c'$ must map $v$ to a vertex with the same colour. Hence by Lemma \ref{lem:nofixedpoint} every such automorphism preserves $c$. \aed We have thus proved the following.

\begin{prp} \label{prp:Not_bipartite_p_odd}
If $G$ is not bipartite and $p$ is odd, then any automorphism which preserves the canonical edge colouring must also preserve $c$. \qed
\end{prp}
\zed

We now define an equivalence relation on a subset of the edge set of $G$. Let $E_C$ be the set of edges of $G$ which are contained in some cycle. Define a relation on $E_C$ by
\[
	e \sim f \Leftrightarrow \text{every cycle containing $e$ also contains $f$}.
\]

\begin{prp}
The relation $\sim$ is an equivalence relation.
\end{prp}

\begin{proof}
The relation clearly is reflexive and transitive. To show that it is symmetric assume for a contradiction that every cycle containing $e$ must also contain $f$, but there is a cycle $C$ which contains $f$ but not $e$. Let $D$ be any cycle containing $e$. Then $(C \cup D) - \{f\}$ contains a cycle through $e$ which does not contain $f$.
\end{proof}

Every equivalence class of the relation $\sim$ is a subset of some cycle (in fact it is the intersection of all cycles containing any and hence all of its members). In particular, every class forms either a cycle or is a disjoint union of paths.
\aed
A {\em path equivalence class} (resp. {\em cycle equivalence class}) is a path (resp. cycle) $P$ in $G$ such that the edges in $P$ are an entire equivalence class for the above relation. An equivalence class that is a disjoint union of (at least two) paths is a {\em disjoint equivalence class}.
\zed

\begin{lem}
\label{lem:existence-pathclass}
\aed
Every cycle which is not a cycle equivalence class contains at least two distinct path equivalence classes.
\zed
\end{lem}

\begin{proof}
We first claim that different equivalence classes cannot \emph{cross} on a cycle, i.e.\ if $e \sim e' \nsim f \sim f'$, then the cyclic order of the four edges on a cycle cannot be $(e,f,e',f')$. Indeed, if $e$ and $e'$ are not equivalent to $f$, then there is a cycle containing $e$ and $e'$, but not $f$. In particular, there is a path $P$ between endpoints of $e$ and $e'$ not containing $f$. If there is a cycle $C$ such that the cyclic order on $C$ is $(e,f,e',f')$, then $P \cup C$ contains a cycle which contains $f'$ but not $f$, showing that $f \nsim f'$.

Now let $C$ be a cycle  \aed that is not a cycle equivalence class. \zed
If all equivalence classes contained in $C$ are 
\aed path equivalence classes, \zed then we are done since there are at least two of them. So assume that there is a \aed disjoint equivalence class $A$ in $C$. \zed Since equivalence classes cannot cross, every other equivalence class is contained in a unique component of $C - A$. \bed Let $X$ be a component of $C - A$ and let $B$ be a disjoint equivalence class contained in $X$. Then at least one component of $X - B$ is a path whose endpoints are not incident with any edge in $A$. Assume that $B$ be an equivalence class which minimises the length of the shortest such path. Any equivalence class that meets this shortest path has to be completely contained in it. \zed By the minimality condition on $B$ each such equivalence class must be connected. \bed Whence every component of $C - A$ contains at least one path equivalence class\zed .
\end{proof}

We now distinguish cases by the different types of equivalence classes: First we deal with the case where there is a \aed cycle equivalence class. If there is no cycle equivalence class, then by Lemma~\ref{lem:existence-pathclass} there is a path equivalence class. We then \zed consider the case where there are path equivalence classes of length at least $2$ and finally we look at the case where all path equivalence classes have length $1$.

\begin{clm}
Theorem~\ref{thm:vertextoedgecolouring} is true if there is a cycle equivalence class.
\end{clm}

\begin{proof}
Let $C$ be a cycle equivalence class. First assume that $C$ has odd length. If $p$ is odd, then we already know that the canonical edge colouring has the desired properties, so we can assume that $p=2$. 

\aed We define an edge colouring $c''$ on $G$ as follows.
For all edges $e \in C$ let $v_e$ be the vertex opposite to $e$ in $C$. If $\sum_{v \in C} c(v) = 1$, then define $c''(e) = c(v_e)$ for all $e \in C$; otherwise set $c''(e) = c(v_e)+1$ for all $e \in C$. For all other edges in $G$ set $c''$ equal to $c'$. \zed
With this definition, the sum of the edge colours of $C$ is $1$. From \ref{obs:odd} we deduce that $C$ is the only \aed odd \zed cycle with this property. In particular, every $c''$-preserving automorphism $\gamma$ must preserve $C$ setwise. Furthermore, such an automorphism must preserve the vertex colouring $c$ on $C$ since equal edge colours of $e,f \in C$ imply equal vertex colours of $v_e$ and $v_f$. Hence $\gamma$ preserves the canonical edge colouring on $C$ and thus on all of $G$. By Lemma \ref{lem:nofixedpoint} any automorphism which preserves $c$ on $C$  and $c'$ on $G$ must preserve $c$ on all of $G$.

So we may assume that the length of $C$ is even. If $C$ contains at least $6$ edges, then 
\aed we instead take $c''$ to be the edge colouring of $G$ in which the edges of $C$ are coloured by the sequence $(1,1,0,1,0,0,\ldots,0)$ while all other edges in $G$ are coloured according to $c'$. Then $C$ is \zed the only cycle violating \ref{obs:even}. Hence 
\aed every $c''$-preserving automorphism $\gamma$ must fix $C$ setwise, \zed and since the sequence gives a distinguishing colouring, $C$ is fixed pointwise. \aed Thus, by Lemma \ref{lem:nofixedpoint} such an automorphism must preserve $c$ on all of $G$.\zed

Finally consider the case where $C$ has length $4$. If $C$ is the only cycle \bed equivalence class \zed of length $4$, then it must be setwise fixed by every automorphism. It is an easy exercise to find an edge colouring $c''$ of $C$ such that every automorphism of $C$ which preserves $c''$ also preserves the vertex colouring on $C$. Hence in this case we can again apply Lemma \ref{lem:nofixedpoint} to conclude that we have found a colouring with the desired properties.

We may hence assume that there is more than one cycle \bed equivalence class \zed of length $4$. Let $D$ be another such cycle. Note that removing the edges of $C$ disconnects $G$ into four components, since otherwise $C$ wouldn't form its own equivalence class. There is a unique vertex $v_D$ of $C$ which lies in the same component as $D$. Analogously we define the vertex $v_C \in D$. Now if we colour $C$ by the sequence $(1,0,0,0)$ and $D$ by the sequence $(1,1,1,0)$, then $C$ and $D$ are the only cycles violating \ref{obs:even}, so both of them must be fixed setwise by any automorphism which preserves the resulting edge colouring. This implies that $v_C$ and $v_D$ must be fixed by each such automorphism, and consequently every automorphism which preserves the edge colouring must fix $C$ and $D$ pointwise. By the same arguments as before we conclude that every automorphism preserving this edge colouring must also preserve $c$.
\end{proof}

\begin{clm}
Theorem~\ref{thm:vertextoedgecolouring} is true if there is a path equivalence class of length at least $2$.
\end{clm}

\begin{proof}
Let $P$ be an arbitrary path equivalence class of length at least $2$, and let $\tilde{c}$ be the colouring obtained by colouring all vertices of $P$ with colour $0$ \aed and all others by $c$. \zed

Let $\tilde{c}'$ be the canonical edge colouring corresponding to $\tilde{c}$ and let $\tilde{c}''$ be the colouring obtained from $\tilde{c}'$ by adding $1$ to the colour of the first edge of $P$. 

Then a cycle violates \ref{obs:even} or \ref{obs:odd} with respect to $\tilde{c}''$ if and only if it contains $P$. Since $P$ is the unique equivalence class with this property, any automorphism which preserves $\tilde{c}''$ must fix $P$ setwise. By Lemma~\ref{lem:nofixedpoint}, every automorphism which fixes $\tilde{c}''$ must thus fix $\tilde{c}$.

Furthermore, since all vertices of $P$ hat the same colour in $\tilde{c}$, all edges had the same colour in $\tilde{c}'$. Since we only changed the colour of the first edge, this implies that $P$ must be fixed pointwise by every automorphism which preserves $\tilde{c}''$. Since the only vertices where $c$ and $\tilde{c}$ potentially differ are those of $P$, this implies that every automorphism which fixes $\tilde{c}''$ must also fix $c$.
\end{proof}

\begin{clm} \label{clm:Path_classes_One}
Theorem~\ref{thm:vertextoedgecolouring} is true if all path equivalence classes have length $1$.
\end{clm}

\begin{proof}
Denote by $E_P$ the set of edges contained in path equivalence classes. Clearly every automorphism must fix $E_P$ setwise. \aed Let us call an edge {\em $c$-monochromatic} if its vertices both receive the same colour in $c$.

We begin with an observation: if there is $c$-monochromatic edge $e\in E_P$, then (using the same colours as $c$) we can describe an edge-colouring $c''$ that is less symmetric than $c$. Indeed, suppose that $e \in E_P$ is $c$-monochromatic. \zed Let $c'$ be the canonical edge colouring corresponding to $c$ and let $c''$ be obtained from $c'$ by adding $1$ to the colour of $e$. Then a cycle violates \ref{obs:even} or \ref{obs:odd} with respect to $c''$ if and only if it contains $e$. Since $e$ is the unique equivalence class with this property, any automorphism which preserves $c''$ must fix $e$. Since the endpoints of $e$ have the same colour in $c$ we conclude by Lemma~\ref{lem:nofixedpoint} that any such automorphism preserves $c$, whence $c''$ is less symmetric than $c$.

\aed
From now on we assume that $E_P$ contains no $c$-monochromatic edges.
We next consider the case where there exists a 3-cycle in $E_P$. If $p = 2$ then there would be a $c$-monochromatic edge in $E_P$, which cannot happen by assumption. On the other hand, if $p$ is odd then $G$ is not bipartite and we can apply Proposition~\ref{prp:Not_bipartite_p_odd} to deduce that the canonical edge colouring corresponding to $c$ is less symmetric than $c$. Hence, if $E_P$ contains a $3$-cycle then our claim holds.

From now on we assume that $E_P$ contains no $3$-cycles. Let us now address the case where there are two distinct incident edges $e = uv$ and $f= vw$ in $E_P$ such that $c(u) = c(w)$. By assumption, $uw \not \in E_P$.
Let $\tilde{c}$ be a new vertex colouring of $G$ in which $\tilde{c}(v) := c(w)$ and for all vertices $x \not = v$ we have $\tilde{c}(x) := c(x)$. Take $\tilde{c}'$ to be the canonical edge colouring of $G$ corresponding to $\tilde{c}$. 
Let $\tilde{c}''$ be obtained from $\tilde{c}'$ by adding $1$ to the colour of $e$. Then a cycle violates \ref{obs:even} or \ref{obs:odd} with respect to $\tilde{c}''$ if and only if it contains $e$. Therefore, if $\gamma$ is a $\tilde{c}''$-preserving automorphism then $\gamma$ must fix $e$ setwise. Since $\tilde{c}''$ and $\tilde{c}'$ differ only in $e$, the automorphism $\gamma$ must also preserve $\tilde{c}'$. Since $\tilde{c}(v) = c(w) = c(u)$ and $\tilde{c}(u) = c(u)$, we have that $\gamma$ preserves $\tilde{c}$ by Lemma~\ref{lem:nofixedpoint}. Now $uv$ and $vw$ are $\tilde{c}$-monochromatic edges and no vertex other than $u,v$ or $w$ lies in a $\tilde{c}$-monochromatic edge in $E_P$. Since $uw \not \in E_P$, we conclude that $v$ is fixed by $\gamma$. Since $\tilde{c}$ and $c$ differ only at $v$, it follows that $\gamma$ preserves $c$. Hence $\tilde{c}''$ is less symmetric than $c$, and our claim holds. 

From now on we assume that if there are two distinct incident edges $e = uv$ and $f= vw$ in $E_P$, then $c(u) \not = c(w)$. Now consider the general case where there are two distinct incident edges $e = uv$ and $f= vw$ in $E_P$. By assumption we have that $c(u) \not = c(w)$ and $uw \not \in E_P$.
Let $\tilde{c}$ be a new vertex colouring of $G$ in which $\tilde{c}(u) := c(w)$ and $\tilde{c}(v) := c(w)$ and for all vertices $x$ distinct from $u$ and $v$ we have $\tilde{c}(x) := c(x)$. Take $\tilde{c}'$ to be the canonical edge colouring of $G$ corresponding to $\tilde{c}$. 
Let $\tilde{c}''$ be obtained from $\tilde{c}'$ by adding $1$ to the colour of $f$. Then a cycle violates \ref{obs:even} or \ref{obs:odd} with respect to $\tilde{c}''$ if and only if it contains $f$. Therefore, if $\gamma$ is a $\tilde{c}''$-preserving automorphism then $\gamma$ must fix $f$ setwise. It follows then that $\gamma$ must also preserve $\tilde{c}'$. 
Since $f$ is $\tilde{c}$-monochromatic, we have that $\gamma$ preserves $\tilde{c}$ by Lemma~\ref{lem:nofixedpoint}. Notice that all edges in $E_P$ that do not contain $u$ or $v$ are not $\tilde{c}$-monochromatic. Furthermore, if there is a $\tilde{c}$-monochromatic edge in $E_P$ not containing $u$ or $w$, say $vu'$, then $\tilde{c}(u') = c(u')$ must equal $\tilde{c}(v) = c(w)$. Hence $wv$ and $vu'$ are two distinct incident edges in $E_P$ with $c(w) = c(u')$, which cannot happen by assumption. Therefore, the only $\tilde{c}$-monochromatic edges in $E_P$ are $uv$ and $vw$ and, possibly, some other edges containing $u$. Since $\gamma$ preserves $\tilde{c}$ and fixes $f$ setwise, it must fix $v$ and therefore it must also fix $w$. Furthermore, $e$ and $f$ are the only $\tilde{c}$-monochromatic edges containing $v$, so $u$ must also be fixed by $\gamma$. Since $\tilde{c}$ and $c$ differ only at $u$ and $v$, it follows that $\gamma$ preserves $c$. Hence $\tilde{c}''$ is less symmetric than $c$ and our claim holds.

From now on we assume no two edges in $E_P$ are incident.
We claim that we can find $e,f \in E_P$ with the following property. \zed

\begin{enumerate}[label=(\textasteriskcentered)]
\item \label{clm:star} There is no automorphism which fixes $e$ and $f$, swaps the endpoints of $e$ or $f$, and preserves the colouring \aed $c$ \zed  on all vertices not incident to $e$ or $f$.
\end{enumerate}

Once we have found such edges, let $\tilde{c}$ be the colouring obtained from $c$ by changing the colour of the endpoints of $e$ to $0$ and the colour of the endpoints of $f$ to $1$. Since $e$ and $f$ are the only \aed $\tilde{c}$-monochromatic edges in $E_P$, \zed  they must be fixed by any \aed $\tilde{c}$-preserving automorphism. \zed By \ref{clm:star} every \aed $\tilde{c}$-preserving \zed automorphism must fix the endpoints of both $e$ and $f$ and thus $\tilde{c}$ is less symmetric than $c$. Since \aed there are $\tilde{c}$-monochromatic edges we can apply the observation we made at the beginning of our proof of Claim~\ref{clm:Path_classes_One}, using $\tilde{c}$ instead of $c$. Hence there is an edge colouring which is less symmetric than $\tilde{c}$ and thus also less symmetric than $c$.\zed  

It remains to show that we can find edges satisfying \ref{clm:star}. Since there are no cycle equivalence classes, every cycle in $G$ contains at least one edge in $E_P$, hence $G - E_P$ is a forest. No edge of $E_P$ has both its endpoints in the same tree $T$, otherwise we would \aed obtain a cycle in $G$ that only contains \zed  one path equivalence class, \aed and this is impossible by Lemma~\ref{lem:existence-pathclass}. \zed

\aed Recall that, by assumption, no two edges in $E_P$ are incident. \zed
Let $e_0 = u_0v_0 \in E_P$, let $T$ be the tree in $G-E_P$ containing $u_0$, and let $C_1$ and $C_2$ be two cycles \aed in $G$ \zed  that both contain $e_0$ but differ in the other edge incident to $u_0$. Such cycles exist because otherwise there would be an edge equivalent to $e_0$ contradicting the assumption that $e_0$ forms its own equivalence class. Assuming that $C_i$ traverses $e_0$ from $v_0$ to $u_0$, let $e_i = u_iv_i$ be the next edge in $E_P$ visited by $C_i$ for $i \in \{1,2\}$.

Denote by $d$ the distance between vertices in $G-E_P$, allowing the value $\infty$ for vertices in different components. 

If there are \aed $i,j \in \{0,1,2\}$ \zed such that $d(u_i,u_j) \neq d(v_i,v_j)$, then an automorphism that swaps $u_i$ and $v_i$ must move $u_j$ to a vertex $x$ with $d(x,v_i) = d(u_j,u_i)$. But $x = u_j$ is impossible since $v_j \notin T$ and thus $d(u_j,v_i) = \infty$, and $x = v_j$ is impossible since $d(u_i,u_j) \neq d(v_i,v_j)$. Hence there is no automorphism which swaps $u_i$ and $v_i$ and fixes $e_j$. Analogously, there is no automorphism which swaps $u_j$ and $v_j$ and fixes $e_i$. Hence $e = e_i$ and $f = e_j$ satisfy~\ref{clm:star}.

Finally assume that $d(u_i,u_j) = d(v_i,v_j)$ for all $i,j$. Since $u_0$ lies on the unique path connecting $u_1$ and $u_2$ in $T$. we know that $d(u_1,u_2) = d(u_1,u_0) + d(u_0,u_2)$. The same equation hence also holds for $v_0$, $v_1$, and $v_2$ showing that $v_0$ lies on the unique path connecting $v_1$ and $v_2$ in some other tree. Now an automorphism swapping $u_1$ and $v_1$ and fixing $e_2$ would have to swap these two paths. Hence such an automorphism would have to swap $u_0$ and $v_0$, but since $c(u_0) \neq c(v_0)$ this means that it would not preserve the colouring. Since an analogous statement is true for an automorphism swapping $u_2$ and $v_2$ and fixing $e_1$, we conclude that in this case $e = e_1$ and $f = e_2$ satisfy \ref{clm:star}.
\end{proof}

\section{Distinguishing colourings}

As mentioned \aed previously, \zed  our main results \aed can \zed be used to characterise the graphs for which equality holds in Theorem~\ref{thm:distindexdistnumber}. In \cite{alikhani-distindexdistnumber}, the following family $\mathcal T$ of trees is introduced. A tree $T$ is contained in \aed $\mathcal T$ \zed if it satisfies the following requirements:

\begin{enumerate}
\item $T$ is bicentred with central edge $e$, 
\item the rooted trees $T_1$ and $T_2$ attached to the endpoints of $e$ are isomorphic, and
\item $T_1$ has a unique distinguishing $k$-edge colouring up to isomorphism.
\end{enumerate}

It is then proved that the finite members of this class are the only finite trees achieving equality in Theorem~\ref{thm:distindexdistnumber}. We extend this result and show that a graph $G$ satisfies Theorem~\ref{thm:distindexdistnumber} with equality if and only if $G \in \mathcal T$.

The following example shows how to construct infinitely many members of $T$.

\begin{exmp}
Let $\{T_i\mid i \in I \}$ be a (finite or infinite) collection of (finite or infinite) rooted trees all of which admit a distinguishing $k$-colouring. Assume further that each $T_i$ only admits finitely many non-isomorphic distinguishing $k$-colourings and let $n_i$ be the number of such colourings. 

Take a disjoint union of an edge $e=uv$ and $2 k n_i$ copies of each $T_i$. For every $i$ draw an edge from $u$ to the roots of $k n_i$ copies of $T_i$. For the remaining copies of $T_i$ draw an edge from $v$ to the root.

Figure~\ref{fig:bad} shows a tree obtained by this construction for $k=2$ obtained in this way, where the collection consisted of a single vertex and a path of length $1$.
\end{exmp}

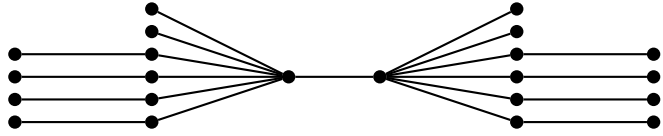
\begin{figure}
\begin{center}
\begin{tikzpicture}
	\GraphInit[vstyle=simple]
	\tikzset{VertexStyle/.style = {shape = circle, fill = black, inner sep = 0pt, minimum size = 5pt}}
	\begin{scope}[rotate=90, xscale = .3, yscale=.6]
	\grEmptyPath[prefix=b,RA=1,RS=7]{4}
	\grEmptyPath[prefix=a,RA=1,RS=4]{6}
	\Vertex[x=2,y=1]{x0}
	\Vertex[x=2,y=-1]{y0}
	\grEmptyPath[prefix=c,RA=1,RS=-4]{6}
	\grEmptyPath[prefix=d,RA=1,RS=-7]{4}
	\EdgeFromOneToAll{x}{a}{0}{6}
	\EdgeFromOneToAll{y}{c}{0}{6}
	\EdgeIdentity*{a}{b}{0,...,3}
	\EdgeIdentity*{c}{d}{0,...,3}
	\EdgeIdentity*{x}{y}{0}
	\end{scope}
\end{tikzpicture}
\end{center}
\caption{An element of $\mathcal T$.}
\label{fig:bad}
\end{figure}

\begin{thm}
Let $T$ be a bicentred tree, let $e=uv$ be its central edge, and let $T_1$ and $T_2$ be the components of $T-e$. Then for any integer $k \geq 2$ the following are equivalent.
\begin{enumerate}[label=(\roman*)]
\item \label{itm:distnumberdistindex} 
$D(T) = k$ and $D'(T)= k+1$.
\item \label{itm:centraledge} 
$D(T) = k$ and there is no distinguishing $k$-vertex colouring $c$ with $c(u) = c(v)$.
\item \label{itm:tbad} 
$T \in \mathcal T$.
\end{enumerate}
\end{thm}

\begin{proof}
We first show that \ref{itm:distnumberdistindex} $\implies$ \ref{itm:centraledge}. If there was a distinguishing $k$-vertex colouring with $c(u) = c(v)$, then (by swapping colours), there is such a colouring with $c(u) = c(v) = 0$, and thus by Lemma~\ref{lem:colouringbijection} also a distinguishing $k$-edge colouring.

For the implication \ref{itm:centraledge} $\implies$ \ref{itm:tbad} first observe that if $T_1$ and $T_2$ are not isomorphic, then every automorphism would fix $u$ and $v$. Thus in any distinguishing $k$-vertex colouring we can change the colours of $u$ and $v$ to obtain a distinguishing vertex colouring with $c(u) = c(v)$. Further, if there is no distinguishing vertex colouring with \aed $c(u) = c(v)$, \zed  then any two \aed distinguishing \zed vertex colourings of $T_1$ with $c(u) = 0$ are isomorphic. Clearly, this implies that their canonical edge colourings must be isomorphic as well and by Lemma~\ref{lem:colouringbijection} there are no further distinguishing $k$-edge colourings of $T_1$.

Finally, for \ref{itm:tbad} $\implies$ \ref{itm:distnumberdistindex}, note that if $T_1$ has an (up to isomorphism) unique distinguishing $k$-edge colouring, then every $k$-edge colouring of $T$ is either not distinguishing in $T_1$ or $T_2$, or has a colour preserving automorphism swapping $T_1$ and $T_2$. Hence $D'(T) > k$. On the other hand, Lemma~\ref{lem:colouringbijection} implies that there is a distinguishing vertex colouring of $T_1$ with $c(u) = 0$ and by changing the colour of the root we obtain $k$ non-isomorphic distinguishing $k$-vertex colourings. By colouring $T_1$ and $T_2$ by two non-isomorphic distinguishing vertex colourings we get a distinguishing vertex colouring of $T$, whence $D(T) \leq k$. By Theorem \ref{thm:distindexdistnumber} the only possibility for $D'(T) > k$ and $D(T) \leq k$ is $D(T) = k$ and $D'(T)= k+1$. 
\end{proof}

\bibliographystyle{abbrv}
\bibliography{ref}

\end{document}